\documentclass{amsart}
\usepackage{amssymb,amsfonts}
\usepackage{latexsym}
\usepackage{amsmath,mathrsfs}
\usepackage[all,arc]{xy}
\usepackage{enumerate}
\usepackage[english]{babel}
\usepackage[bookmarks=true]{hyperref}
\usepackage{tikz}
\usepackage{blkarray}
\usepackage{mathtools}
\newtheorem{theo}{Theorem}[section]
\newtheorem{coro}[theo]{Corollary}
\newtheorem{prop}[theo]{Proposition}
\newtheorem{lemm}[theo]{Lemma}

\theoremstyle{definition}
\newtheorem{defi}[theo]{Definition}

\theoremstyle{remark}
\newtheorem{rema}[theo]{Remark}

\newcommand{\bb}[1]{\mathbb{#1}}
\newcommand{\al}[1]{\mathcal{#1}}
\newcommand{\sr}[1]{\mathscr{#1}}
\newcommand{\ak}[1]{\mathfrak{#1}}
\newcommand{\gen}[1]{\left\langle #1\right\rangle}
\newcommand{\sm}{\smallsetminus}

\newcommand{\ra}{\rightarrow}
\newcommand{\lra}{\longrightarrow}
\newcommand{\x}[1]{\text{#1}}

\begin{document}
\title{On very stablity of  principal  $G-$bundles}
\author{Hacen ZELACI}
\address{Mathematical Institute of the university of Bonn.}
\curraddr{}
\email{z.hacen@gmail.com}
\date{\today}
\subjclass[2010]{Primary 14H60, 14H70.}

\begin{abstract}
	Let $X$ be a smooth irreducible projective curve. In this notes, we generalize the main result of \cite{PA} to principal $G-$bundles for any semisimple linear algebraic group $G$. After defining  very stability of principal $G-$bundles, we show that this definition is equivalent to the fact that  the Hitchin fibration restricted to the space of Higgs fields on that principal bundle is finite. We also study the relation between  very stability and  other stability conditions in the case of $\x{SL}_2-$bundles. 
\end{abstract}
\maketitle
\tableofcontents
\section{Introduction}
Let $X$ be a smooth irreducible projective curve over $\bb C$ of genus $g\geqslant 2$. Denote  its canonical bundle by $K_X$.  Let $E$ be a stable vector bundle of degree $0$ and rank $r$  over $X$.
Let $\sr H_E$ be the Hitchin map:$$\sr H_E:H^0(X,E\otimes E^*\otimes K_X)\lra W:=\bigoplus_{i=1}^r H^0(X,K_X^i).$$
defined by associating to a Higgs field  $\phi\colon E\ra E\otimes K_X$ the coefficients of its characteristic polynomial $$((-1)^i\x{Tr}(\Lambda^i\phi))_{i=1,\dots,r}\in W.$$  
Following \cite{LG}, the vector bundle $E$ is called very stable if $\sr H_E^{-1}(0)=\{0\}$. In other words, the vector bundle $E$ is very stable if and only if  it has no nilpotent Higgs field other than $0$.  With these notations, the following theorem has been proven recently in \cite{PA}. 
\begin{theo}\label{PA}
	The vector bundle $E$ is very stable if and only if $\sr H_E$ is finite.
\end{theo}
We will give a general and more elementary proof for this. \\

Let $G$ be a semisimple connected linear algebraic group. In a natural way, a principal $G-$bundle $E$ over $X$ is called \emph{very stable}  if the fiber of the Hitchin map (see Section \ref{pre}) over $0$ is reduced to $0$. We will show that this is equivalent to say that the bundle $\x{ad}(E)\otimes K_X$ has non nilpotent section with respect to the adjoint action.\\
Our main result is then the following 
\begin{theo}\label{main}
	  The $G-$bundle $E$ is very stable if and only if $\sr H_E$ is finite.
\end{theo}

Our proof is purely algebraic and is independent of the geometry of the moduli space of semistable $G-$Higgs pairs. In particular it induces an elementary  proof of Theorem \ref{PA}. Actually we will show the following 
\begin{theo} \label{main2}
	Let $f:\bb A^n\ra \bb A^n$ be a morphism given by homogeneous polynomials such that $f^{-1}(0)=\{0\}$, then $f$ is finite. \\
\end{theo}

I want to thank  C. Pauly  for suggesting this problem and for useful remarks that considerably improve this notes, and D. Huybrechts for useful discussions. I also want to thank I. Grosse-Brauckmann and A. Pe\'on-Nieto. \\

\section{Preliminaries}\label{pre}  We recall here the definitions and the basic facts about principal Higgs bundles and the Hitchin fibration. 

Let $G$ be a semisimple linear  algebraic group. A principal $G-$bundle $E$ over $X$ is a smooth variety $E\ra X$ with a free (right) $G-$action which is equivariant with respect to the trivial action on $X$ such that it is locally trivial in the \'etale topology, i.e. there exists a covering  $(U_i)$ of $X$ such $E|_{U_i}\cong U_i\times G$ with the right standard $G-$action. 

Let $E$ be a $G-$bundle. Given a finite dimensional representation $\rho\colon G\ra \x{GL}(V)$, $E$ induces a vector bundle $$E(V)=E\times_G V\coloneqq E\times V/G,$$ for the diagonal action of $G$ on $E\times V$. Using this, we say that $E$ is semistable if there exists a  (equivalently for any) \emph{faithful} representation   $\rho\colon G\ra GL(V)$ such that $E(V)$ is a semistable vector bundle over $X$.

Let $P\subset G$ be a subgroup. A reduction of the structure group of the $G-$bundle $E$ to $P$ is a section $\sigma:X\ra E(G/P)$. The pullback $\sigma^*E$ of $E$ via $\sigma$ is a $P-$bundle over $ X$ which is also called abusively a reduction of the structure group of $E$ to $P$.  Remark that $\sigma^*E(G)\cong E$. 

Ramanathan in his thesis, gave the following definition of semistability: 
\begin{defi}
	The $G-$bundle $E$ is semistable if for any maximal parabolic subgroup $P\subset G$ and for any reduction of the structure group  $\sigma:X\lra E(G/P)$, one has $$\x{deg}(\sigma^*T_{E(G/P)})\geqslant 0,$$
	where $T_{E(G/P)}=E(\ak{g}/\ak{p})$ is the tangent bundle. We say that $E$ is stable if the above inequality is strict. 
\end{defi}
Actually the two definitions are equivalent
\begin{theo}[\cite{R1}, \cite{R2}]
	The $G-$bundle $E$ is semistable if and only if it is semistable in the sense of Ramanathan.
\end{theo}
Note that this is not true if we replace semistable with stable.  \\ 

Consider now the Lie algebra $\ak g:=\x{Lie}(G)$. The adjoint bundle $\x{ad}(E)$ is defined to be  the vector bundle $E(\ak{g})$ associated to the adjoint representation of $G$ on $\ak g$. Since $G$ is semisimple, $\x{ad}(E)$ is a self-dual vector bundle over $X$ of Lie algebras isomorphic to $\ak{g}$. Moreover, the adjoint representation is faithful, hence $E$ is semistable if and only if $\x{ad}(E)$ is semistable vector bundle. \\ 
Following \cite{HM}, we say that $E$ is \emph{ad-stable} if $\x{ad}(E)$ is stable vector bundle. 
\begin{prop}[\cite{HM}]
 If the $G-$bundle $E$ is ad-stable, then it is stable (in the sense of Ramanathan).
\end{prop}
Note also that there are no ad-stable bundles for a non semisimple algebraic group $G$ (see $loc.cit$).\\

\section{Very stability and  finiteness of the Hitchin map}

Let $Q_1,\dots, Q_m$ be a basis of the algebra of invariant polynomials on $\ak{g}$ such that $Q_i$ is homogeneous of  $d_i$. This basis defines a map called the Hitchin morphism \begin{equation}\label{1}\sr H_E:H^0(X,\x{ad}(E)\otimes K_X)\lra W_{G}:=\oplus_{i=1}^m H^0(K_X^{d_i}).\end{equation}

A $G-$Higgs pair is a pair $(E,\phi)$, with $\phi\in H^0(X,\x{ad}(E)\otimes K_X)$. We call $(E,\phi)$ \emph{semistable} if the Higgs bundle $(\x{ad}(E),\phi)$ is semistable. \\ 
Moreover, by \cite{NH}, we have 
\begin{prop}
With the assumption of semi-simplicity of  $G$ we have 	 $$\emph{dim}(H^0(X,\emph{ad}(E)\otimes K_X))=\emph{dim}(W_{G}).$$
\end{prop}
Let $\al M_X(G)$ be the moduli space of semistable $G-$Higgs pairs $(E,\phi)$. The Hitchin map gives a fibration $$\sr H:\al M_X(G)\lra W_{G}.$$ 
By Faltings \cite{F}, $\sr H$ is proper. However, in general $\sr H_E$ is not necessarily proper (take $E$ to be the trivial $G-$bundle for example), and this is due to the fact that  the canonical map $$H^0(X,\x{ad}(E)\otimes K_X)\lra \al M_X(G)$$ is not  proper in general.\\

\begin{defi}
	A $G-$bundle $E$ over $X$ is called \emph{very stable} if the fiber of $\sr{H}_E$ over $0$ is equal to $\{0\}$.
\end{defi} 
Let $\ak{n}\subset \ak{g}$ be the nilpotent cone. Since $\ak{n}$ is $G-$invariant,  it makes sense to consider sections contained  in $\ak{n}$. Such sections are called nilpotent (\cite{BR}).  So the above definition is equivalent to say that the only nilpotent section of $\x{ad}(E)\otimes K_X$ is the zero one. Indeed, let $$\al S(\ak{g}^*)$$ be the symmetric algebra on $\ak{g}^*$. Then the ring  $\al S(\ak{g}^*)^G$ of invariant polynomials on $\ak{g}$ with respect to the adjoint action  is generated by \emph{trace polynomials}: $\ak{g}\ra \bb C$ $x\mapsto \x{Tr}(\x{ad}(x)^k)$ for $x\in \ak{g}$ and $k\geqslant 0$. So  we see that sections in the fiber of $\sr{H}_E$ over $0$ are exactly those being in  the nilpotent cone $\ak n$. For a detailed proof of this fact see \cite[Proposition 16]{Ko}. \\

\begin{prop}
	A very stable $G-$bundle over $X$ is stable. The locus of very stable $G-$bundles is a nonempty open subset of $\al M_X(G)$.
\end{prop} 
\begin{proof}
	The proof is given in  \cite[Proposition $5.2$]{BR}. However for the completeness, we give a detailed  argument. Note firstly that this is not true over $\bb P^1$ (see \cite{LG}).

	Let's treat the case of vector bundles firstly.  Consider a non stable vector bundle $E$ of degree $0$. So there exist  vector bundles $F$ and $H$ such that $\x{deg}(F)=-\x{deg}(H)\geqslant 0,$  and  an exact sequence $$0\ra F\ra E\ra H\ra 0.$$ 
	Using Riemann-Roch theorem one deduces that $H^0(H^*\otimes F\otimes K_X)\not=0$. Let $\phi'$ be a non zero section in that vector space, then $\phi'$ gives a non trivial map $G\ra F\otimes K_X$. Finally define $\phi$ to be the composition $$E\ra H\ra F\otimes K_X\ra E\otimes K_X. $$ Clearly  $\phi$ is a non trivial nilpotent Higgs field of $E$. 
	
	Assume now that the $G-$bundle $E$ is not stable.   
	Let ${P}\subset G$ be a maximal parabolic subgroup, put $\ak{p}=\x{Lie}(P)$. Denote by $F=\sigma^*E$  a reduction of the structure group of $E$ with respect to some section $\sigma$. Note that $F(G)=E$ and $F(\ak{g})=\x{ad}(E)$. \\
	Let $\ak g/\ak p$ be the isotropy representation of $P$. Then $\sigma^*T_{E(G/P)}= F(\ak{g}/\ak{p})$ and by assumption $$\x{deg}(F(\ak{g}/\ak{p}))\leqslant 0.$$ Now let $\ak{p^\perp}$ be the orthogonal of $\ak p$ with respect to the Killing form. It is actually the nilpotent radical of $\ak p$. The representation of $P$ on $\ak p^\perp$ is the dual to the isotropy representation, hence the vector bundle $F(\ak{p}^\perp)$ has degree $\geqslant 0$. This implies by Riemann-Roch that $H^0(F(\ak{p}^\perp)\otimes K_X)$ is non zero. But we have an exact sequence $$0\ra \ak{p}\ra \ak{g}\ra \ak{g}/\ak{p}\ra 0.$$ Taking the dual we get the exact sequence $$0\ra \ak{p}^\perp\ra \ak{g}^*\ra \ak{p}^*\ra 0.$$ Since $\ak{g}^*\cong \ak g$, we deduce that  $F(\ak{p}^\perp)\hookrightarrow \x{ad}(E)$.  This implies that $E$ is not very stable.
\end{proof}

This proof implies that if 	 $E$ is  a very stable $G-$bundle, then for any maximal parabolic subgroup $P$ of $G$, and for any reduction of the structure group $\sigma: X\ra E(G/P)$, we have 
$$\x{deg}(\sigma^*T_{E(G/P)})\geqslant r(g-1),$$ where $r=\x{dim}(\ak{g}/\ak{p})=\x{rk}(\sigma^*T_{E(G/P)})$.\\

Recall that the group $\bb C^*$ acts on $H^0(\x{ad}(E)\otimes K_X)$  by $(\lambda,\phi)\mapsto \lambda\phi.$ 
Consider the weighted action of $\bb C^*$  on $W_{G}$ given by $$\lambda\cdot(s_i)_i= (\lambda^{d_i}s_i),$$ where $s_i\in H^0(X,K_X^{d_i})$. Then the  Hitchin map $\sr{H}_E$ is equivariant with respect to these actions.\\


Now we come to our main theorem. 

\begin{theo}
  The $G-$bundle $E$ is very stable if and only if $\sr H_E$ is finite.
\end{theo}

	Note that when $\sr H_E$ is finite then it is quasi-finite. Hence $E$ is very stable because the fiber over $0$ is stable under the $\bb C^*-$action. 
	

It is also not difficult to show that $\sr H_E$ is quasi-finite when $E$ is assumed to be very stable. Indeed  assume that $E$ is very stable.  Consider the $\bb C^*-$action on both spaces in (\ref{1}), $\sr{H}_E$ is equivariant with respect to these actions. Consider the map from $W_{G}$ to $\mathbb N\cup\{-\infty\}$ given by $$s\lra\x{dim}(\sr{H}_E^{-1}(s)).  $$  This map is upper semi-continuous (\cite[Th\'eor\`eme 13.1.3]{EGA}). In particular its restriction to each $\bb C^*-$orbit is again upper semi-continuous.  From this and since $0$ is in the closure of any $\bb C^*-$orbit,   we deduce that the dimension of each fiber is smaller or equal to the dimension of the special fiber over $0$. But $E$ is very stable, thus by definition $\x{dim}(\sr{H}_E^{-1}(0))=0$. It follows that $\x{dim}(\sr H_E^{-1}(s))\leqslant 0$ for any $s\in W$. Hence $\sr H$ is quasi-finite. However, it is still not clear why $\sr H_E$ should be finite. \\

The following Theorem is a slight generalization  of our main result.
\begin{theo}\label{2}
	Let $n\geqslant 1$ and $f:\bb A^n\ra \bb A^n$ be a morphism given by homogeneous polynomials such that $f^{-1}(0)=\{0\}$. Then $f$ is finite. In particular, it is proper.
\end{theo}
\begin{proof}
	Denote $f=(P_1,\cdots,P_n)$ and let $X_1,\cdots,X_n$ be the coordinates on $\bb A^n$. Let $d_i=\x{deg}(P_i)\geqslant1$. 
	For  $k=1,\cdots,n,$ and for $i=1,\dots,n$,  consider  the polynomials $$P_i^k=P_i(T_1,\cdots,1,\cdots, T_n),$$ where $1$ is at the $k^{th}$ position and $T_i=X_i/X_k$. Since $f^{-1}(0)=\{0\}$, the polynomials $P_1^k,\dots,P_n^k$ have no common zero in $\bb A^{n-1}$. Hence the ideal generated by them contains $1$.  Let $U_i^k$ be some polynomials in $T_i$'s such that $$\sum_{i=1}^nU_i^kP_i^k=1.$$ Take $c$ to be a positive  integer bigger than  $d_i+\x{deg}(U_i^k)$ for all $i$ and all $k$. Then we see that $$ X_k^c=\sum_{i=1}^n \tilde{U}_i^kP_i,$$ where $\tilde{U}_i^k=X_k^{c-d_i}U_i^k$. In particular, since $c>d_i+\x{deg}(U_i^k)$, we have $$\x{deg}(\tilde{U}_i^k)=c-d_i<c.$$  Let $A=\bb C[P_1,\cdots,P_n]$ and $M=\bb C[X_1\cdots,X_n]$.  For $\alpha=(\alpha_1,\cdots,\alpha_n)\in \bb N^n$, let  $X^\alpha=X_1^{\alpha_1}\cdots X_n^{\alpha_n}$ be the corresponding  monomial in $M$. We claim that the $A-$module $M$ is generated by the finite set  $$S:=\{X^\alpha\; |\,\alpha \in\bb N^n \;\x{ such that } \alpha_i<c\x{ for all }i\}. $$ 
	Let $N$ be the $A-$module generated by $S$. Clearly $N\subset M$. To prove the converse inclusion, we will show, by induction on the total degree $|\alpha|=\sum_i\alpha_i$, that $X^\alpha \in N$ for any $\alpha\in \bb N^n$.\\
	For $\alpha\in \bb N^n$ such that  $|\alpha|< c$, we have by definition $X^\alpha\in S$. Hence the basis of the induction. \\ 
	Now, let $m$ be an integer such that $m\geqslant c$ and assume that for all $\alpha\in \bb N^n$ such that $|\alpha|<m$ we have $X^\alpha\in N$. Let $\alpha\in \bb N^n$ such that $|\alpha|=m$. If for all $i$, $\alpha_i< c$ then $X^\alpha\in S$ and we are done. Otherwise, there exists $k\in \{1,\cdots,n\}$ such that $\alpha_k\geqslant c$. Let $\beta=(\alpha_1,\cdots,\alpha_k-c,\cdots,\alpha_n)$. Using the above decomposition of $X_k^c$ we deduce that $$ X^\alpha=\sum_{i=1}^n X^\beta\tilde{U}_i^kP_i.$$  
	But the polynomials $X^\beta \tilde{U}_i^k$ are homogeneous of degrees $$|\alpha|-c+\x{deg}(\tilde{U}_i^k)< |\alpha|=m.$$ It follows by the induction hypothesis that for all $i$,  $X^\beta\tilde{U}_i^k \in N$. Hence  $X^\alpha\in N$. This ends the proof.
\end{proof}
\begin{rema}
		\begin{enumerate}
			\item Using Hilbert zero theorem, we see that actually the condition $f^{-1}(0)=\{0\}$ implies that $\ak{m}^N\subset \gen{P_i}$, where $\ak m=\gen{X_1,\dots, X_n}$. Hence $X_k^N\in \gen{P_i}$ for each $k$. But using this, one can not control the degrees of the coefficients $\tilde{U}_i^k$.
			\item	The condition that  the polynomials $P_i$ are homogeneous  means that $f$ is equivariant with respect to an appropriate $\bb C^*-$actions on the base and target.
		\end{enumerate}
\end{rema}
Since the Hitchin morphism $\sr H_E$ is of finite presentation, then $\sr H_E$ is  finite if and only if it is  proper, hence we deduce the following  
\begin{coro}
The $G-$bundle 	$E$ is very stable if and only if $\sr H_E$ proper. $\hfill\square$
\end{coro}
Another consequence is the following
\begin{coro}
	Let $E$ be a very stable $G-$bundle, then $\sr H_E$ is surjective. In particular, for any very stable vector bundles $E$, if $q:\tilde{X}_s\ra X$ is the spectral curve associated to a general spectral data $s\in W_{\x{GL}_r}$, then there exists a line bundle $L$ over $\tilde{X}_s$ such that $q_*L\cong E$.
\end{coro}
\begin{proof}
	For the second part, use \cite[Proposition $3.6$]{BNR}.
\end{proof}

Note that the Hitchin map is never injective. To see this let $d$ be the degree of $Q_l$ one of the generators of the ring of invariants. Let $s=(s_1,\cdots,s_m)\in W_G$ be such that $s_l\not=0$ and $s_i=0$ for all $i\not=l$. Then assuming that $E$ is very stable, by the surjectivity of the Hitchin map, there exists a Higgs field $\phi$ such that $\sr H_E(\phi)=s$. But then $\sr{H}_E(\xi\phi)=s$ for any $d^{th}$ root of unity $\xi\in \bb C$. \\

\section{Ad-stable vs very stable $\x{SL}_2-$bundles}
In this section, we study the relation between very stability and ad-stability  of $\x{SL}_2-$bundles over $X$. We will show that there is no implication between these two notions. \\
In the following, we denote by $\al U_X(2,0)$ the moduli space of semistable vector bundle of rank $2$ and degree $0$, and by $\al{SU}_X(2)$ the moduli space of vector bundle with trivial determinant.

Let $E$ be an $\x{SL}_2-$bundle, we denote by $E_v$ the associated vector bundle which has trivial determinant. Note that $\x{ad}(E)=\x{End}_0(E_v)$ the Lie algebra bundle of traceless endomorphisms of $E_v$.  Hence $E$ is ad-stable if and only if $\x{End}_0(E_v)$ is stable. We assume hereafter that $E$ (equivalently $E_v$) is semistable. \\
\begin{lemm}\label{adstable}
	If the $G-$bundle $E$  is ad-stable, then $E_v$ is a stable vector bundle.
\end{lemm}
\begin{proof}
	We know that $E_v$ is semistable because $E$ is semistable  $G-$bundle (actually it is stable by \cite{HM}). Assume that $E_v$ is not stable. Let $\eta\in \x{Pic}^0(X)$ such that $$0\ra \eta\ra E_v\ra \eta^{-1}\ra 0. $$ Let $\nu : \eta^{-1}\ra \eta\otimes\eta^{-2}$ be the canonical map. Then the composition $$E_v\ra \eta^{-1}\stackrel{\nu}{\lra} \eta\otimes\eta^{-2}\ra E_v\otimes\eta^{-2}$$ defines a non zero map $\phi:E_v\ra E_v\otimes\eta^{-2}$ which is nilpotent, i.e. $\phi^2=0$. Hence $$\x{Tr}(\phi)=0. $$
	So we deduce that $H^0(\x{End}_0(E_v)\otimes\eta^{-2})\not=0$, or, in other words, $\eta^2\hookrightarrow \x{End}_0(E_v)$. Hence $\x{End}_0(E_v)$ is not stable. So $E$ is not ad-stable.
\end{proof}

Let $\eta\in J_X[2]$ be a non trivial $2-$torsion line bundle over $X$. Denote $q:X_\eta\ra X$ the associated \'etale double cover. Note that the map $$q_*:\x{Pic}^0(X_\eta)\lra\al U_X(2,0)$$ is defined everywhere.  We define $\al S_\eta\subset \al U_X(2,0)$ to be the image of $q_*$. 
 Then we have 
\begin{prop}
	The $\x{SL}_2-$bundle $E$ is ad-stable if and only if $E_v$ is stable and  does not belong to the set $$\al S:=\bigcup_{\eta\in J_X[2]\sm\{\al O_X\}}\al S_\eta.$$
\end{prop}
\begin{proof}
	If $E_v$ is not stable then $E$ is not ad-stable by Lemma \ref{adstable}.  
	So assume $E_v$ is stable. Now, $\x{End}_0(E_v)$ is semistable and $H^0(\x{End}_0(E_v))=0$. Since it has rank three, $\x{End}_0(E_v)$ is not stable if and only if there exists a line bundle $\eta\in J_X$ such that $H^0(\x{End}_0(E_v)\otimes \eta)\not=0$. But since $\eta\not=\al O_X$ by the above, we have $$H^0(\x{End}_0(E_v)\otimes \eta)=H^0(\x{End}(E_v)\otimes \eta).$$
So we deduce, from the stability of $E_v$, that these last spaces are non zero if and only if $E_v\cong E_v\otimes \eta$. Hence taking the determinant we get $\eta^2=\al O_X$. From \cite{NR}, we deduce that $E_v\cong q_*L$ for some line bundle $L\in J_{X_\eta}$,  hence $E_v\in \al S_\eta$. \\
The converse is straightforward.
\end{proof}
 
Now let $\kappa$ be a theta characteristic and denote by  $\Theta_\kappa\subset \al{SU}_X(2,0)$  the  associated theta divisor given by $$\Theta_{\kappa} \coloneqq\{E\in \al{SU}_X(2)\mid h^0(E\otimes \kappa)\not=0\}.$$ Then we have 
\begin{prop}
The divisor  $\Theta_{\kappa}$  is included in the complement of the locus of very stable vector bundles.
\end{prop}
\begin{proof}
	 For  $E\in \Theta_{\kappa}$,  we have an exact sequence  $$0\ra L^{-1}\ra E\ra L\ra 0.,$$ where $L^{-1}$ is the image of $\kappa^{-1}$ in $E$. Since $H^0(X,L^{-1}\otimes\kappa)\not=0$, we deduce that $H^0(X, L^{-2}\otimes K_X)\not=0$. Let $s$ be a non zero global section of $L^{-2}\otimes K_X$. Then the composition $$E\twoheadrightarrow \kappa\stackrel{s}{\lra} \kappa^{-1}\otimes K_X\hookrightarrow E\otimes K_X$$ is a non trivial Higgs field which is clearly nilpotent. This shows  that $E$ is not very stable. 
\end{proof}
	   It was pointed to me by Pauly that this result has been already known, see \cite[and references therein]{PP}.
\begin{coro} There exists an ad-stable $\x{SL}_2-$bundle which is not very stable.
\end{coro}
 \begin{proof}
 Since the locus $\Theta_{\kappa}\subset\al{SU}_X(2)$ is a divisor,  so its dimension is $3g-4$, but we see that $\al S\cap \Theta_\kappa$ inside $\al U_X(2)$ has dimension at most $g-1$. Indeed, consider the double cover  $q:X_\eta\ra X$  associated to some $\eta$.  Then the vector bundle $q_*L$ has trivial determinant if and only if $\x{Nm}(L)=\eta$, where $\x{Nm}:\x{Pic}(X_\eta)\ra \x{Pic}(X)$ is the norm map. Hence $\al S\cap\Theta_\kappa$  is a finite union (over $\eta$)  of the direct image by $q_*$ of the intersection of $\x{Nm}^{-1}(\eta)$ and the Riemann theta divisor in $\tilde{\Theta}_{q^*\kappa}\subset \x{Pic}^0(X_\eta)$.  But $\x{dim}(\x{Nm}^{-1}(\eta))=g-1$, hence the claim.\\ Moreover, since $g\geqslant2$, we have $g-1<3g-4$.  In particular, there exists a vector bundle $E\in\Theta_\kappa\sm\al S$.  So $E$ is ad-stable, but not very stable.
 \end{proof}
Conversely, we show the existence of a very stable  $SL_2-$bundle which is not ad-stable.
\begin{prop}
	There exists a very stable $\x{SL}_2-$bundle which is not ad-stable.
\end{prop}
 \begin{proof} Let  $E$ be a non very stable vector bundle in $\al{SU}_X(2)$, and let $\phi$ be a non zero nilpotent Higgs field. If $L^{-1}=\x{ker}(\phi)$ then we get the following diagram $$\xymatrix{0 \ar[r] &  L^{-1}\ar[r] &  E\ar[r]^p &  L\ar[r]\ar[dll]_\nu  &0  \\ 0\ar[r] & L^{-1}\otimes K_X\ar[r]^i&  E\otimes K_X\ar[r] &  L\otimes K_X\ar[r]& 0}$$
and $\phi=i\circ \nu\circ p$, where $\nu$ is a non zero global section of $L^{-2}\otimes K_X$. This implies that $\x{deg}(L)\leqslant  g-1$ and $\x{deg}(L)\geqslant0$ since $E$ is semistable.\\
 
 	Let $\eta\in J_X[2]$ non trivial, and let $q:X_\eta\ra X$ the associated unramified double cover. 
     For $d=0,\dots,g-1$, let $\Theta_d\subset \x{Pic}^d(X)$ (resp. $\Xi_{2d}\subset \x{Pic}^{2d}(X_\eta)$) be the locus of line bundles $L$ such that $K_X\otimes L^{-2}$ (resp. $L$) has a non zero global section. The locus $\Xi_{2d}$ is called the Brill-Noether locus in $\x{Pic}(X_\eta)$ of degree $2d$ and its dimension  is $2d$ (for $d\leqslant g-1$). While the quotient of the locus $\Theta_d$  by $\x{Pic}^0(X)[2]$ is isomorphic to the Brill-Noether locus in $\x{Pic}(X)$  of degree $2g-2d-2$, hence its dimension  is  given by 
     $$\x{dim}(\Theta_d)= \begin{cases}  2g-2d-2 & \x{if } 2d\geqslant g-1 \\ g & \x{if } 2d \leqslant g-2.
     \end{cases}$$ 
     Now, consider the map $$\Psi_d:\Xi_{2d}\times \Theta_d\lra \x{Pic}^0(X_\eta)$$ that associates to $(M,N)$ the line bundle $M\otimes q^*N^{-1}$. These maps are never surjective. Indeed, if $2d\leqslant g-2$, we have $$\x{dim}(\x{Im}(\Psi_d))\leqslant \x{dim}(\Xi_{2d}\times \Theta_d)=2d+g<2g-1.$$ 	Otherwise $ 2d\geqslant g-1$, and in this case we also have $$\x{dim}(\x{Im}(\Psi_d))\leqslant \x{dim}(\Xi_{2d}\times \Theta_d)=2d+2g-2d-2=2g-2<2g-1.$$ 
     
      %
 	 Moreover, the union of the images of $\Psi_d$, for $d=0,\dots,g-1$,  is certainly not the whole of  $\x{Pic}^0(X_\eta)$. Let $L\in \x{Pic}^0(X_\eta)$ not in the union of the  images of $\Psi_d$ and such that $E'=q_*L$ is stable (note that such line bundle exists  by \cite[Proposition $6.2$]{Z}). Now let $E=E'\otimes \delta^{-1}=q_*(L\otimes q^*\delta^{-1})$ where $\delta$ is a line bundle such that $\delta^2=\x{det}(E')$.  Then $E$ is not ad-stable because $E\in \al S$, and by the remark in the beginning of this proof, it is very stable because it has no  line subbundle $L$ of degree $-0,\dots,-(g-1)$ such that $K_XL^{-2}$ has non zero global section. This ends the proof.
 \end{proof}


\vspace{1cm}

\bibliographystyle{alpha}
\bibliography{bib}
\end{document}